\documentclass[a4paper,10pt]{article}

\usepackage{lineno} 

\usepackage{graphicx}
\usepackage{geometry}
\usepackage{amsthm}
\usepackage{amssymb}
\usepackage{amsmath}
\usepackage{hyperref}
\usepackage{xcolor}
\usepackage{enumerate}
\usepackage{complexity}
\usepackage{listings}

\usepackage[font=small,labelfont=bf]{caption}
\usepackage[font=small,labelfont=normalfont,labelformat=simple]{subcaption}

\usepackage{todonotes}

\graphicspath{{figs/}}

\newtheorem{theorem}{Theorem}
\newtheorem{definition}{Definition}
\newtheorem{lemma}[theorem]{Lemma}
\newtheorem{corollary}[theorem]{Corollary}
\newtheorem{question}{Question}

\DeclareMathOperator{\conv}{conv}
\DeclareMathOperator{\sgn}{sgn}

\def\GG{{\cal G}}
\def\HH{{\cal H}}
\def\NN{{\mathbb N}}

\author
{
Manfred Scheucher
\and
Hendrik Schrezenmaier
\and
Raphael Steiner
}

\date{}

\title{A Note On Universal Point Sets for Planar Graphs\footnote{
An extended abstract of this work 
was presented at the 35th European Workshop on Computational Geometry (EuroCG'19) \cite{3S_eurocg_version}.
A short version 
is to appear in the Proc.\ of the 27th International Symposium on
Graph Drawing and Network Visualization (GD'19) \cite{3S_gd_version}.
Earlier versions of this paper (EuroCG 2019; arXiv versions 1 and 2) contained a flaw, 
which has been corrected. 
For more details, see Section~\ref{sec:thebug}.
}}

\begin{document}
\maketitle

\begin{center}
{\footnotesize
Institut f\"ur Mathematik, Technische Universit\"at Berlin, Germany,\\
\texttt{\{scheucher,schrezen,steiner\}@math.tu-berlin.de}
}
\end{center}

\vspace{0.5cm}

\begin{abstract}
We investigate which planar point sets allow 
simultaneous straight-line embeddings of all planar graphs on a fixed number of vertices.
We first show that at least $(1.293-o(1))n$ points are required 
to find a straight-line drawing of each $n$-vertex planar graph 
(vertices are drawn as the given points);
this improves the previous best constant $1.235$ by Kurowski (2004).

Our second main result is based on exhaustive computer search:
We show that no set of 11 points exists, on which
all planar 11-vertex graphs can be simultaneously drawn plane straight-line.
This strengthens the result by Cardinal, Hoffmann, and Kusters (2015),
that all planar graphs on $n \le 10$ vertices can be 
simultaneously drawn on particular \emph{$n$-universal} sets of $n$ points 
while there are no $n$-universal sets of size $n$ for $n \ge 15$.
We also provide 49 planar 11-vertex graphs 
which cannot be simultaneously drawn on any set of 11 points.
This, in fact, is another step towards a (negative) answer of the question,
whether every two planar graphs can be drawn simultaneously -- a question raised by
Brass, Cenek, Duncan, Efrat, Erten, Ismailescu, Kobourov, Lubiw, and Mitchell (2007).
\end{abstract}

\section{Introduction}
\label{sec:intro}

A point set $S$ in the Euclidean plane is called
\emph{$n$-universal for} a family $\GG$ of planar $n$-vertex graphs 
if every graph $G$ from $\GG$ admits a 
plane straight-line embedding
such that the vertices are drawn as points from~$S$.
A point set, which is $n$-universal for the family of all planar graphs, is simply called \emph{$n$-universal}.
We denote 
by $f_p(n)$ the size of a minimal $n$-universal set (for planar graphs), and
by $f_s(n)$ the size of a minimal $n$-universal set for stacked triangulations, 
where stacked triangulations (a.k.a.\ planar 3-trees) are defined as follows:
\begin{definition}[Stacked Triangulations]
Starting from a triangle, one may obtain any stacked triangulation 
by repeatedly inserting a new vertex inside a face (including the outer face) of the current triangulation 
and making it adjacent to all the three vertices contained in the face.
\end{definition}
An example of a stacked triangulation is shown in Figure~\ref{fig:stacked_triang_example}.

\begin{figure}[htb]
\centering
\includegraphics[scale=0.75]{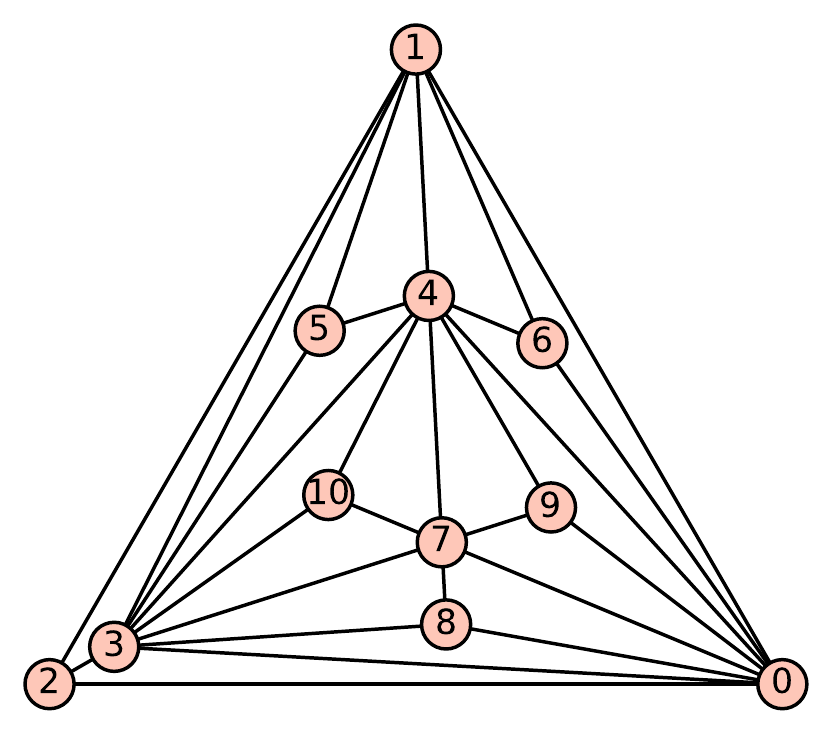}
\caption{A (labeled) stacked triangulation on~$11$ vertices
in which every face is incident to a degree-3-vertex.}
\label{fig:stacked_triang_example}
\end{figure}

\medskip

De~Fraysseix, Pach, and Pollack \cite{DeFraysseixPachPollack1990} showed
that every planar $n$-vertex graph  
admits a straight-line embedding on a $(2n-4) \times (n-2)$ grid 
-- even if the combinatorial embedding (including the choice of the outer face) is prescribed. Moreover, the graphs are only embedded on a triangular subset of the grid.
Hence, $f_p(n) \le n^2-O(n)$.
This bound was further improved to the currently best known bound $f_p(n) \leq \frac{n^2}{4}-O(n)$ \cite{BannisterCDE2014} 
(see also \cite{Schnyder1990,Brandenburg2008}).
Also various subclasses of planar graphs have been studied intensively:
Any stacked triangulation on $n$~vertices (with a fixed outer face) 
can be drawn on a particular set of $f_s(n) \le O(n^{3/2} \log n)$ points \cite{FulekToth2013}.
For outerplanar graphs, it is known that any set of $n$ points in general position is $n$-universal \cite{PachGritzmannMoharPollack1991,CastanedaUrrutia1996}.
An upper bound of $O(n \log n)$ is known for 2-outerplanar graphs and for simply nested graphs,
and an $O(n \cdot \polylog (n))$ bound is known for graphs of bounded pathwidth~\cite{Angelini2018,BannisterCDE2014}.

\smallskip
Concerning the lower bound on $f_p(n)$ and $f_s(n)$, respectively, 
the obvious relation $n \le f_s(n) \le f_p(n)$ holds for any $n \in \NN$.
The first non-trivial lower bound on the size of $n$-universal sets was also given 
by de~Fraysseix, Pach, and Pollack \cite{DeFraysseixPachPollack1990}, who showed a lower bound of $f_p(n) \ge n+(1-o(1))\sqrt{n}$.
Chrobak and Karloff~\cite{Chrobak1989} further improved the lower bound to $(1.098-o(1))n$, 
and the multiplicative constant was later on improved to $1.235$ by Kurowski \cite{Kurowski2004}.
In fact, Kurowski's lower bound even applies to~$f_s(n)$.

Cardinal, Hoffmann, and Kusters \cite{CardinalHoffmannKusters2015} 
showed that $n$-universal sets of size $n$ exist for every $n \le 10$,
whereas for $n \ge 15$ no such set exists -- not even for stacked triangulations.
Hence $f_p(n) = f_s(n) = n$ for $n \le 10 $
and $f_p(n) \ge f_s(n) > n$  for  $n \geq 15$.
Moreover, they found a collection of 7,393 planar graphs on $n=35$ vertices 
which cannot be simultaneously drawn straight-line on a common set of $35$ points.
We call such a collection of graphs a \emph{conflict collection}.
This was a first big step towards an answer to the question
by Brass and others \cite{BrassCDEEIKLM2007}:
\begin{question}[\cite{BrassCDEEIKLM2007}]
\label{question:conflict_collection_size2}
Is there a conflict collection of size~2?
\end{question}

\section{Outline}

Our first result is the following theorem, 
which further improves the lower bound on $f_s(n)$.
We present its proof in Section~\ref{sec:proof_theorem_lower129}.

\begin{theorem}\label{thm:lower129}
It holds that $f_s(n) \ge (\alpha-o(1))n$, where $\alpha=1.293\ldots$ is 
the unique real-valued solution of the equation ${\alpha^\alpha} \cdot {(\alpha-1)^{1-\alpha}}=2$.
\end{theorem}

In Section~\ref{sec:proof_conflict_11} we present our second result,
which is another step towards a (negative) answer of Question~\ref{question:conflict_collection_size2}
and strengthens the results from \cite{CardinalHoffmannKusters2015}.
Its proof is based on exhaustive computer search.

\begin{theorem}[Computer-assisted]
\label{thm:conflict_11}
There is a conflict collection consisting 
of 49 stacked triangulations on 11 vertices.
Furthermore, there is no conflict collection consisting 
of 36 triangulations on 11 vertices. 
\end{theorem}

\begin{corollary}
\label{cor:no_uni_11}
There is no 11-universal set of size 11 -- even for stacked triangulations.
Hence, $f_p(11) = f_s(11) = 12$.
\end{corollary}

\noindent
The equality in Corollary~\ref{cor:no_uni_11} is witnessed 
by an 11-universal sets of 12 points (cf.\ Listing~\ref{lst:n12univ11}). 
\begin{lstlisting}[
caption={An 11-universal set of $12$ points.},
captionpos=b,label=lst:n12univ11,basicstyle=\small,breaklines=true]
 [(214,0),(0,13),(2,16),(9,26),(124,12),(133,11),(148,9),(213,1),(211,4),(210,6),(116,179),(122,197)]
\end{lstlisting}

Last but not least, since all known proofs for lower bounds 
make use of separating triangles,
we also started the investigation of 4-connected triangulations.
In Section~\ref{sec:4connected} we present some $n$-universal sets of size~$n$ 
for 4-connected planar graphs for all $n \le 17$.

A detailed description of our tools can be found in 
Section~\ref{sec:tools}, and the edge list of the 49 stacked triangulationsgraphs of the conflict collection 
can be found in Section~\ref{sec:conflict_collection}.

\section{Proof of Theorem~\ref{thm:lower129}}
\label{sec:proof_theorem_lower129}

To prove the theorem, we use a refined counting argument 
based on a construction of a set of labeled stacked triangulations 
that was already introduced in \cite{CardinalHoffmannKusters2015}.
There it was used to disprove the existence of $n$-universal sets of $n \ge 15$ points 
for the family of stacked triangulations.

\goodbreak
\begin{definition}[Labeled Stacked Triangulations, {cf.~\cite[Section~3]{CardinalHoffmannKusters2015}}]
For every integer $n \ge 4$, we define the family $\mathcal{T}_n$ of labeled stacked triangulations 
on the set of vertices $V_n:=\{v_1,...,v_n\}$ inductively as follows:  
\noindent
\begin{itemize}
\item $\mathcal{T}_4$ consists only of the complete graph $K_4$ with labels $v_1,\ldots,v_4$.
\item If $T$ is a labeled graph in $\mathcal{T}_{n-1}$ with $n \ge 5$, and $v_iv_jv_k$ defines a face of $T$, 
then the graph obtained from $T$ by stacking the new vertex $v_n$ to $v_iv_jv_k$ 
(i.e., connecting it to $v_i$, $v_j$, and~$v_k$) is a member of $\mathcal{T}_n$.
\end{itemize}
\end{definition}
\goodbreak

It is important to notice that, when speaking of $\mathcal{T}_n$, 
we distinguish between elements if they are distinct as \emph{labeled graphs}, 
even if their underlying graphs are isomorphic. 
The essential ingredient we will need from \cite{CardinalHoffmannKusters2015} is the following.

\begin{lemma}[cf.\ {\cite[Lemmas~1 and~2]{CardinalHoffmannKusters2015}}] 
\label{lem:stackedproperties}
\noindent
\begin{enumerate}[(i)]
\item \label{lem:stackedproperties_item_i}
For any $n \ge 4$, the family $\mathcal{T}_n$ contains exactly $2^{n-4}(n-3)!$ labeled stacked triangulations. 
\item \label{lem:stackedproperties_item_ii}
Let $P_n=\{p_1,\ldots,p_n\}$ be a set of $n \ge 4$ labeled points in the plane. 
Then for any bijection $\pi:V_n \rightarrow P_n$, 
there is at most one $T \in \mathcal{T}_n$ such that 
the embedding of~$T$, which maps each vertex $v_i$ to point $\pi(v_i)$, defines a straight-line-embedding of $T$.
\end{enumerate}
\end{lemma} 
Figure~\ref{fig:stacked_triang_example} illustrates the idea of item~(\ref{lem:stackedproperties_item_ii}) of Lemma~\ref{lem:stackedproperties}.

\medskip
We need the following simple consequence of the above:
\begin{corollary} \label{cor:consequences_stackedproperties}
Let $P=\{p_1,\ldots,p_m\}$ be a set of $m \ge n \ge 4$ labeled points in the plane. 
Then for any injection $\pi:V_n \rightarrow P$, 
there is at most one $T \in \mathcal{T}_n$ such that 
the embedding of~$T$, which maps each vertex $v_i$ to point $\pi(v_i)$, defines a straight-line-embedding of $T$.
\end{corollary}
\begin{proof}
Let $T_1,T_2 \in \mathcal{T}_n$ be two stacked triangulations 
such that $\pi$ describes a plane straight-line embedding of both. 
Since $\pi$ is an injection, this means that $\pi$ defines 
a straight-line embedding of both $T_1,T_2$
on the sub-point set $Q:=\pi(V_n)$ of $P$ of size $n$. 
Applying Lemma~\ref{lem:stackedproperties}(\ref{lem:stackedproperties_item_ii})
to the bijection $\pi:V_n \rightarrow Q$ and $T_1,T_2$, 
we deduce $T_1=T_2$. This proves the claim.
\end{proof}

We are now ready to prove Theorem~\ref{thm:lower129}.

\begin{proof}[Proof of Theorem~\ref{thm:lower129}.]
Let $n \ge 4$ be arbitrary and $m:=f_s(n) \ge n$. 
There exists an \mbox{$n$-universal} point set $P=\{p_1,\ldots,p_m\}$
for all stacked triangulations,
hence for every $T \in \mathcal{T}_n$
there exists a straight-line embedding  of $T$ on~$P$,
with (injective) vertex-mapping $\pi:V_n \rightarrow P$.
By Corollary~\ref{cor:consequences_stackedproperties}, 
we know that no two stacked triangulations from~$\mathcal{T}_n$ 
(each of which has the same vertex set)
yield the same injection~$\pi$. 
Consequently, by Lemma~\ref{lem:stackedproperties}(\ref{lem:stackedproperties_item_i}), we have
\[
2^{n-4}(n-3)!=|\mathcal{T}_n| \leq \frac{m!}{(m-n)!},
\]
which means
\[
\frac{1}{16n(n-1)(n-2)}2^n \leq \binom{m}{n}=\binom{f_s(n)}{n}.
\]
Let $\beta(n):=\frac{f_s(n)}{n}$. Using the fact that (Stirling-approximation)
\[
\binom{f_s(n)}{n} \sim \underbrace{\sqrt{\frac{f_s(n)}{2\pi n(f_s(n)-n)}}}_{\le 1}\frac{f_s(n)^{f_s(n)}}{n^n(f_s(n)-n)^{f_s(n)-n}}
\leq \left(\frac{\beta(n)^{\beta(n)}}{(\beta(n)-1)^{\beta(n)-1}}\right)^n,
\]
we deduce (taking logarithms) that:
\[
(1-o(1))n \leq \log_2 \left(\frac{\beta(n)^{\beta(n)}}{(\beta(n)-1)^{\beta(n)-1}}\right)n \Longleftrightarrow 2-o(1) 
\leq \frac{\beta(n)^{\beta(n)}}{(\beta(n)-1)^{\beta(n)-1}}.
\]
Consequently, $\beta(n) \ge (1-o(1))\alpha$, 
where $\alpha$ is the unique solution to $ \frac{\alpha^\alpha}{(\alpha-1)^{\alpha-1}}=2$. 
This proves $f_s(n)=n\cdot \beta(n) \ge (1-o(1))\alpha n$, which is the claim.
\end{proof}

\section{Proof of Theorem~\ref{thm:conflict_11} and Corollary~\ref{cor:no_uni_11}}
\label{sec:proof_conflict_11}

In the following, 
we outline the strategy which we have used to find 
a conflict collection of 49 stacked $11$-vertex triangulations.
We refer the reader who is mainly interested in verifying our computational results
 directly to Section~\ref{sec:how_to_verify}.

\medskip

One fundamental observation is the following:
if an $n$-universal point set has collinear points, 
then by perturbation one can obtain another $n$-universal point set \emph{in general position}, 
i.e., with no collinear points.
Hence, in the following we only consider point sets in general position.
Also it is not hard too see that, if two point sets are \emph{combinatorially equivalent}, i.e., 
there is a bijection such that the corresponding triples of points induce the same orientations,
then both sets allow precisely the same straight-line drawings. 
Hence, in the following we further restrict our considerations to \emph{(non-degenerated) order types}, i.e., 
the set of equivalence classes of point sets (in general position).

\subsection{Enumeration of Order Types}
\label{subsec:OrderTypeEnum}

The database of all order types of up to $n=11$ points
was developed by Aurenhammer, Aichholzer, and Krasser \cite{AichholzerAurenhammerKrasser2001,AichholzerKrasser2006}
(see also Krasser's dissertation \cite{Krasser2003}).
The file for all order types of up to $n=10$ points (each represented by a point set) 
is available online,
while the file for $n=11$ requires almost 100GB of storage 
and is available on demand \cite{AichholzerOTDB}.
Their algorithm starts with an \emph{abstract order type} on $k-1$ points 
(which only encodes the triple orientations of a point set), 
computes its dual pseudoline arrangement, 
and inserts a $k$-th pseudoline in all possible ways.
Due to geometrical constraints,
there are in fact abstract order types enumerated which do not have a realization as a point set.
However, since every order type is in fact also an abstract order type,
it is sufficient for our purposes to test all abstract order types 
-- independent from realizability.

\medskip
For means of redundancy and to provide a fully checkable and autonomous proof, 
we have implemented an alternative algorithm 
to enumerate all abstract order types based on the following idea:
Given a set of points $s_1,\ldots,s_n$ with $s_i=(x_i,y_i)$ sorted left to right\footnote{%
in the dual line arrangement the lines are sorted by increasing slope},
and let
\[
\chi_{ijk} := \sgn \det \begin{pmatrix}
                         1  & 1  & 1  \\
                         x_i& x_j& x_k\\
                         y_i& y_j& y_k
                        \end{pmatrix} \in \{-1,0,+1\}
\]
denote the induced triple orientations,
then  the \emph{signotope axioms} assert that, 
for every 4-tuple $s_i,s_j,s_k,s_l$ with {$i<j<k<l$},
the sequence 
\[
\chi_{ijk},\ \chi_{ijl},\ \chi_{ikl},\ \chi_{jkl}
\] 
(index-triples in lexicographic order)
changes its sign at most once. 
For more information on the signotope axioms 
we refer to Felsner and Weil \cite{FelsnerWeil2001} (see also~\cite{BalkoFulekKyncl2015}).

Given an abstract order type on $n-1$ points,
we insert a $n$-th point in all possible ways,
such that the signotope axioms are preserved.
With our C++ implementation, 
we managed to verify the numbers of abstract order types from \cite{AichholzerAurenhammerKrasser2001,AichholzerKrasser2006,Krasser2003}.
In fact, the enumeration of all 2,343,203,071 abstract order types of up to $n=11$ points (cf.~\href{http://oeis.org/A6247}{OEIS/A6247})
can be done within about 20 CPU hours.

\subsection{Enumeration of Planar Graphs}
\label{subsec:GraphEnum}

To enumerate all non-isomorphic maximal planar graphs on 11 vertices (i.e, triangulations),
we have used the plantri graph generator (version~4.5) \cite{BrinkmannMcKay1999}.
It is worth to note that also the nauty graph generator \cite{McKayPiperno2014} can be used for the enumeration
because the number of all (not necessarily planar) graphs on 11 vertices is not too large
and the database can be filtered for planar graphs in reasonable time 
-- negligible compared to the CPU time which we have used for later computations.
For various computations on graphs, 
such as filtering stacked triangulations or to produce graphs for this paper, 
we have used SageMath \cite{sagemath_website}\footnote{%
We recommend the Sage Reference Manual on Graph Theory \cite{sagemath_graphtheory_manual} 
and its collection of excellent examples.}.

\subsection{Deciding Universality using a SAT Solver}

For a given point set $S$ and 
a planar graph $G=(V,E)$
we model a propositional formula in conjunctive normal form (CNF)
which has a solution if and only if $G$ can be embedded on~$S$
-- in fact, the variables encode a straight-line drawing.\footnote{%
Cabello~\cite{Cabello2006} showed that deciding embeddability is \NP-complete in general.
His reduction, however, constructs a 2-connected graph, and therefore
the hardness remains unknown for 3-connected planar graphs.}

To model the CNF, we have used 
the variables $M_{vp}$ to describe whether vertex $v$ is mapped to point~$p$,
and
the variables $A_{pq}$ to describe whether the straight-line segment $pq$ 
between the two points $p$ and $q$ is ``active'' in a drawing.

It is not hard to use a CNF to assert that 
such a vertex-to-point mapping is bijective.
Also it is easy to assert that, 
if two adjacent vertices $u$ and $v$ are 
mapped to points $p$ and $q$,
then the straight-line segment $pq$ is active.
For each pair of crossing straight-line segments $pq$ and $rs$
(dependent on the order type of the point set)
at least one of the two segments is not allowed to be active.

\paragraph*{Implementation detail:}
We have implemented a C++ routine which, given a point set and a graph as input, 
creates an instance of the above described model and then uses the solver MiniSat~2.2.0 \cite{EenSorensen2003}
to decide whether the graph admits a straight-line embedding.

\subsection{Finding Conflict Collections -- A Quantitive Approach}
\label{ssec:QuantitiveApproach}

Before we actually tested whether a set of 11 points is 11-universal or not, 
we discovered a few necessary criteria for the point set, 
which can be checked much more efficiently.
These considerations allowed a significant reduction of the total computation times.

\paragraph*{Phase 1:}
There are various properties that a universal point set has to fulfill:

Property~1: The planar graph depicted in Figure~\ref{fig:graph3sep} 
asserts an 11-universal set $S$ -- if one exists -- to have a certain structure.
If the embedding is as on the left of Figure~\ref{fig:graph3sep}, 
then one of the two degree~3 vertices is drawn as extremal point of~$S$, i.e., 
lies on the boundary of the convex hull $\conv(S)$ of~$S$.
After the removal of this particular point, 
the remaining 10 points have 4 convex layers of sizes 3, 3, 3, and 1, respectively.
If the embedding is as on the right of Figure~\ref{fig:graph3sep}, 
then either one or two points of the blue triangle are drawn as extremal points of~$S$ 
(recall the triangular convex hull of~$S$).
And again, the points inside the blue triangle and outside 
the blue triangle have convex layers of sizes 3, 3, 1, and 3, 1, respectively.

\begin{figure}[htb]
\centering
\includegraphics{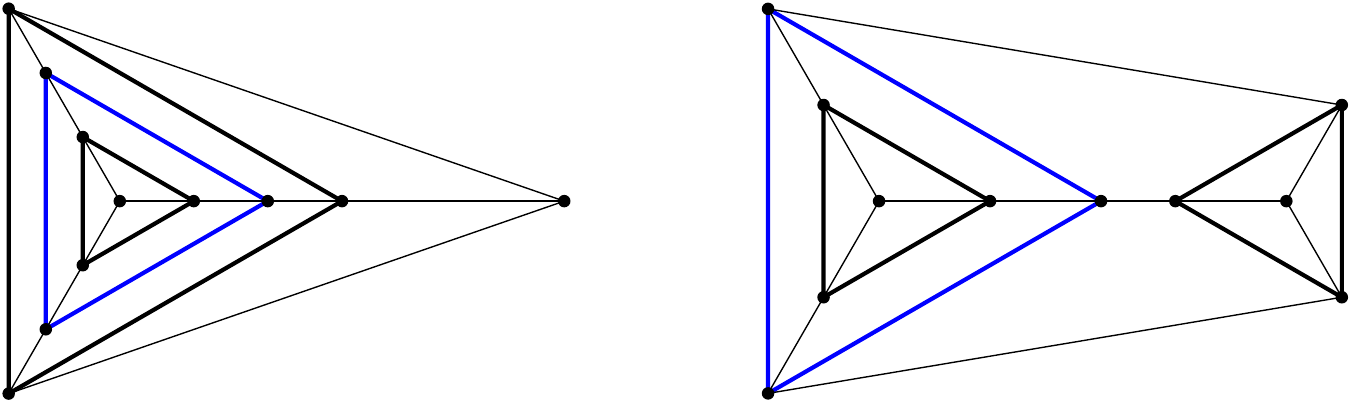}
\caption{The two embeddings of a graph, which forces the point set to have a certain structure.
Each of the vertices of the blue triangle connects to one of the vertices of the two copies of $K_4$.}
\label{fig:graph3sep}
\end{figure}

Property~2: 
There exist a stacked triangulations on 11 points
in which every face is incident to a degree-3-vertex;
see for example Figure~\ref{fig:stacked_triang_example}.
Independent from the embedding of this graph,
there is a degree-3-vertex on the outer face,
and hence all inner points lie inside a triangle spanned by an interior point and two extremal points.
In particular, such a point set must have a triangular convex hull.

Altogether, only 262,386,428 of the 2,343,203,071 abstract order types on 11 points 
fulfill Properties~1 and~2. (The computation time was about 10 CPU hours.)

\paragraph*{Phase 2:}
For each of the 262,386,428 abstract order types on 11 points
which fulfill the conditions above,
we have tested the embeddability of all maximal planar graphs on $n$ vertices separately using
a SAT-solver based approach.
In fact, as soon as one graph was not embeddable, 
the remaining graphs needed not to be checked.
To speed up the computations we have used a priority queue: 
a graph which does not admit an embedding gets increased priority for other point sets to be tested first.

To keep the conflict collection 
as small as possible,
we first filtered out all point sets
which do not allow a simultaneous embedding of all planar graphs on 11 vertices with maximum degree~$10$.
Only 287,871 of the 262,386,428 abstract order types remained (computation time about 100 CPU days).
It is worth to note that there are 82 maximal planar graphs on 11 vertices with maximum degree~10 (cf.~\href{http://oeis.org/A207}{OEIS/A207}),
and that each of these graphs is a stacked triangulation.

At this point one can check with about 10 CPU hours 
that the remaining 287,871 abstract order types are not universal for stacked triangulation on $11$ vertices.
Moreover, since some stacked triangulations on 11 vertices (e.g.\ $G_{12}$ from Listing~\ref{lst:11triangs_1}) 
contain the graph from Figure~\ref{fig:graph3sep} as a subgraph, 
the family of all 434 stacked triangulations on 11 vertices 
(cf.~\href{http://oeis.org/A27610}{OEIS/A27610}) is a conflict collection,
and Corollary~\ref{cor:no_uni_11} follows directly.

\paragraph*{Phase 3:}
To find a smaller conflict collection, 
we tested for each of the 434 stacked triangulations and each of the 287,871 remaining abstract order types,
whether an embedding is possible (additional 35 CPU days).
We used this binary information to formulate an integer program searching for a minimal set of triangulations, without simultaneous embedding.
Using the Gurobi solver (version~8.0.0) \cite{gurobi}, 
we managed to find a collection~$\GG$ of 27 stacked triangulations 
which cannot be embedded simultaneously;
see Listing~\ref{lst:11triangs_1} in Section~\ref{sec:conflict_collection}. 

Since we asserted in Phases~1 and~2 that 
\begin{enumerate}[(1)]
\item \label{property1}
the graph in Figure~\ref{fig:graph3sep}, 
\item \label{property2}
a triangulation where every face is incident to a vertex of degree~3, and 
\item \label{property3}
all 82 triangulations with maximum degree 10
\end{enumerate}
occur in the conflict collection, this yields a conflict collection of size $111=1+1+82+27$.
In fact, since this subset of 27 stacked triangulations contains triangulations 
fulfilling properties (\ref{property1}) and~(\ref{property2}) 
(see, e.g., graphs $G_{12}$ und $G_{10}$ in Listing~\ref{lst:11triangs_1}), 
we indeed have a conflict collection of size~$109$.

\medskip

We have also ran the computations for the collection of all 1,249 triangulations 
(cf.~\href{http://oeis.org/A109}{OEIS/A109}),
and the Gurobi solver showed that any conflict collection 
of (arbitrary) 11-vertex triangulations has size at least~26.

\paragraph*{Phase 4:}
Recall that a minimal conflict collection not necessarily needs to fulfill the properties (\ref{property1})--(\ref{property3}).
Hence we again repeat the strategy from Phase~2, except that we test for the embeddability of 
the 27 stacked triangulations from the collection~$\GG$ obtained in Phase~3 instead of 
the 82 maximal planar graphs on 11 vertices with maximum degree~10.

After another 230 days of CPU time, 
our program had filtered out 2,194 of the 262,386,428 abstract order types (obtained in Phase~1)
which allow a simultaneous embedding of the 27 stacked triangulations from~$\GG$.

\paragraph*{Phase 5:}
As the reader might already guess, 
we proceed as in Phase~3:
we tested for each of the 434 stacked triangulations
and each of the 2,194 order types from Phase~4,
whether an embedding is possible (only some CPU hours).
Using the Gurobi solver, we managed to find 
a collection~$\HH$ of 22 stacked triangulations, which cannot be simultaneously embedded on those order types; 
see Listing~\ref{lst:11triangs_2} in Section~\ref{sec:conflict_collection}. 

Together with the 27 stacked triangulations from~$\GG$
we obtain a conflict collection of size~49, 
and the first part of Theorem~\ref{thm:conflict_11} follows.

\paragraph*{Phase 6:}
To further improve the lower bound,
we have repeated our computations for 
the union of the two sets of point sets obtained in Phase~3 and Phase~5, respectively.
Using Gurobi, 
we obtained that
\begin{itemize}
\item 
any conflict collection of stacked triangulations must have size at least~40, and
\item 
any conflict collection of (arbitrary) triangulations must have size at least~37. 
\end{itemize}

For means of redundancy, we have verified all lower bounds obtained by Gurobi
also using CPLEX (version~12.8.0.0) \cite{IBMstudio}, 
which performed similar to Gurobi. 

\medskip

This completes the proof of the second part of Theorem~\ref{thm:conflict_11}.

\subsection{How to Verify our Results?}
\label{sec:how_to_verify}

To verify the computational results
which are essential for the proof of the first part of Theorem~\ref{thm:conflict_11}, 
one can enumerate all order types on 11 points
and test the conflict collection of 49 triangulations
(\verb|data/triangulations/n11_conflicting49.txt|).
Starting with the unique order type on 3 points 
(\verb|data/order_types/n3_order_types.bin|),
it takes about 1 CPU day to enumerate all order types on 11 points.
By falsifying simultaneous embeddability of the 49 graphs 
(this computation takes about 200 CPU days, but can be run parallelized),
the first part of Theorem~\ref{thm:conflict_11} is then verified.

\medskip
For the second part of the Theorem~\ref{thm:conflict_11},
one can filter the order types, 
which allow a simultaneous embedding of the triangulations from Phase~2 and~4,
and then -- using CPLEX or Gurobi -- 
compute the minimum size of a conflicting collection
among all 11-vertex triangulations and 11-vertex stacked triangulations, respectively.
To save some computation time, we provide the filtered list in the files 
\verb|data/triangulations/n11_after_phase2.bin.zip| and \verb|n11_after_phase4.bin.zip|.
The list of all (stacked) triangulations is provided in the files
\verb|n11_all_triangulations.txt| and 
\verb|n11_all_stacked_triangulations.txt|.

\medskip
A more detailed description is provided in Section~\ref{sec:tools}.
The source codes of our programs and relevant data are 
available on the companion website~\cite{scheucher_website}.

\section{Universal Sets for 4-Connected Graphs}
\label{sec:4connected}

For $n\leq 10$, examples of $n$-universal sets of $n$ points for planar $n$-vertex graphs 
were already given in \cite{CardinalHoffmannKusters2015}.
To provide $n$-universal sets for 4-connected planar graphs for $n=11,\ldots,17$,
we slightly adapted our framework. 
Again, we enumerated 4-connected planar triangulations using the plantri graph generator,
and using our C++ implementation, tested for universality.
Our idea to find the proposed point sets for $n = 11,\ldots,17$
was to start with an $(n-1)$-universal set of $n-1$ points and 
insert an $n$-th point in all possible ways (cf. Section~\ref{subsec:OrderTypeEnum}). 
The abstract order types obtained in this way -- if they turned out to be universal --
were then realized as point sets using our framework \emph{pyotlib}\footnote{%
The ``\textbf{py}thon \textbf{o}rder \textbf{t}ype \textbf{lib}rary'' was initiated during 
the Bachelor's studies of the first author \cite{scheucher2014} 
and provides many features to work with (abstract) order types
such as local search techniques, 
realization or proving non-realizability of abstract order types, 
coordinate minimization and ``beautification'' for nicer visualizations.
For more information, please consult the author.}.
The obtained sets are given in Listing~\ref{lst:4conn_universal}.

\begin{lstlisting}[caption={A set $\{p_1,\ldots,p_{17}\}$ of 17 points such that
$\{p_1,\ldots,p_k\}$ is universal for 4-connected planar $k$-vertex graphs for all $k \in \{11,\ldots,17\}$.},
captionpos=b,label=lst:4conn_universal,basicstyle=\small,breaklines=true]
 [(612,666),(754,635),(415,709),(884,597),(596,695),(890,977),(384,716),(834,609),(424,707),(974, 10),(890,962),(306,805),(301,810),(4,736),(0,735),(975,6),(980,0)]
\end{lstlisting}

It is also worth to note that the numbers of 4-connected triangulations for $n=11,\ldots,20$ are
25; 87; 313; 1,357; 6,244; 30,926; 158,428; 836,749; 4,504,607; 24,649,284  (cf.~\href{http://oeis.org/A7021}{OEIS/A7021}).
Hence, even if a universal point set is known, 
it is getting more and more time consuming to verify $n$-universality
as $n$ gets larger (at least using our SAT solver approach).

\section{Detailed Description of our Tools}
\label{sec:tools}

In following we give a detailed description of the tools which 
are required to verify the proof of Theorem~\ref{thm:conflict_11}. 
Moreover, we exemplify how the tools can be used. 
Even though our C++ code is platform-independent,
we assume that the reader uses a Unix/Linux operating system and 
only give usage examples for this particular setup.
(We ran our experiments in Fedora 27 and openSUSE 15.)

\subsection{Enumerating Abstract Order Types}

\paragraph*{extend\_order\_type}
We provide a C++ program \verb|extend_order_type| 
which reads all abstract order types on a fixed number of points $n$ from the input file,
extends it in all possible ways,
and writes all so-obtained abstract order types on $n+1$ points to an output file (without duplicates).
The program (see the folder \verb|cprogram/scripts/extend_order_type/|)
can be built using qmake\footnote{In fact, no Qt-specific features are used. 
Alternatively to qmake one could also use cmake or just use a stand-alone Makefile} and make.
The following bash command builds the program:
\begin{verbatim}
$ qmake && make
\end{verbatim}

\paragraph*{The File Format}
Concerning the file format, we have to explain a little more theory: 
An abstract order type can be encoded by its triple orientations 
\[
\Lambda_{i,j,k} \in \{-1,0,+1\} \text{ for each } 1 \le i,j,k \le n.
\]
Since encoding this ``big lambda matrix'' uses a cubic amount of bits,
it is more efficient to encode the ``small lambda matrix'', which has the following entries:
\[
\lambda_{i,j} := |\{ k \in \{1,\ldots,n\} \setminus \{i,j\} \colon \Lambda_{i,j,k} > 0 \}|\text{ for each }1 \le i,j,k \le n.
\]
Diagonal elements $\lambda_{i,i}$ are omitted (or can be set to zero).
This data structure was first introduced by Goodman and Pollack~\cite{GoodmanPollack1983}.

Small lambda matrices of (non-degenerated) abstract order types fulfill
$\lambda_{i,j}+\lambda_{j,i} = n-2$ 
(i.e., for each two fixed points $i,j$, 
any other point either lies on the left or on the right side of the directed line through $i$ and $j$),
hence only entries $\lambda_{i,j}$ with $1 \le i<j \le n$ need to be stored.
Moreover, the first point can be assumed to lie on the boundary of the convex hull,
and other points can be assumed to be sorted around the first point.
Such a labeling of points is called ``natural labeling'' and 
yields $\Lambda_{1,i,j} = +$ for $1 < i < j \le n$ and
$\lambda_{1,j} = j-1$ for all $1 < j \le n$.
Consequently, elements from the first row of the small lambda matrix need not to be stored.
Note that the lexicographically minimal small lambda matrix 
(over all labelings) -- 
which we compute to distinguish different order types --
is also naturally labeled.

Altogether, we encoded the entries $\lambda_{i,j}$ for $1 < i<j \le n$ as 8-bit (1-byte) integers in lexicographic order,
i.e, 
\[
\lambda_{2,3},
\lambda_{2,4},
\ldots,
\lambda_{2,n},
\lambda_{3,4},
\lambda_{3,5},
\ldots,
\lambda_{n-1,n}.
\]

For more information we again refer to the articles by Aurenhammer, Aichholzer, and Krasser \cite{AichholzerAurenhammerKrasser2001,AichholzerKrasser2006}, 
and the dissertation of Krasser~\cite{Krasser2003}.

\paragraph*{Usage of the Program}
The program \verb|extend_order_type| can be used as follows:
\begin{verbatim}
./extend_order_type [n] [order types file] [parts] [from part] [to part]
\end{verbatim}
The following describes the parameters.
\begin{itemize}
 \item ``n'' is the number of points in the abstract order types from the input file,
 \item ``order types file'' is the path to the input file,
 \item ``parts'' is the number of threads in total,
 \item ``from part'' is the id of the first thread to be run, and
 \item ``to part'' is the id of the first thread not to be run.
\end{itemize}
The difference ``to part''-``from part'' is precisely the number of threads 
to be started on the local machine. 
As an example, to start a computation on 4 machines with 4 threads each (16 threads in total),
one can simply run one of the following commands on each of the machines:
\begin{verbatim}
./extend_order_type [n] [order types file] 16 0 4
./extend_order_type [n] [order types file] 16 4 8
./extend_order_type [n] [order types file] 16 8 12
./extend_order_type [n] [order types file] 16 12 16
\end{verbatim}
We also provide our python script \verb|create_jobs.py|,
which we used to automatically create job files for the parallel computations on the cluster.

\paragraph*{Complete Enumeration}
To generate all abstract order types,
we start with a binary file \verb|n3_order_types.bin| with content ``0x00'' (one byte)
-- this encodes the unique order type on 3 points, which is described by the small lambda matrix
\begin{verbatim}
-  0  1
1  -  0*
0  1  -
\end{verbatim}
The entry $\lambda_{23}=0$, which is marked with a star (*), is the one entry 
which is actually encoded as ``0x00'' in the file.
This file is also available in the folder \verb|data/order_types/|.

The following command now enumerates all abstract order types on 4 points:
\begin{verbatim}
$ xxd n3_order_types.bin
00000000: 00                                       .
$ ./extend_order_type 3 n3_order_types.bin 1 0 1
n: 3
starting threads: 1
[0/1] started
[0/1] n 3       ct 0    extensions 0
[0/1] finished
all threads done.
total solutions: 2/1
$ xxd n3_order_types.bin.ext0_1.bin 
00000000: 0001 0001 0001
\end{verbatim}
Note that the command xxd displays a hex dump of the given file.
The generated output file \verb|n3_order_types.bin.ext0_1.bin| 
contains 6 bytes in total, encoding the two order types on 4 points (with 3 bytes each).
The first three bytes 
encode the order type of 4 points in convex position, 
which has the following small lambda matrix:
\begin{verbatim}
-  0  1  2
2  -  0* 1*
1  2  -  0*
0  1  2  -
\end{verbatim}
The remaining three bytes encode the other order type of 4 points, 
which has a triangular convex hull and one interior point. 
Its small lambda matrix is the following:
\begin{verbatim}
-  0  1  2
2  -  1* 0*
1  1  -  1*
0  2  1  -
\end{verbatim}
When renaming the file \verb|n3_order_types.bin.ext0_1.bin| to \verb|n4_order_types.bin|, 
one can now analogously enumerate all order types of 5, 6, \ldots points with
\begin{verbatim}
./extend_order_type 4 n4_order_types.bin 1 0 1
./extend_order_type 5 n4_order_types.bin 1 0 1
...
\end{verbatim}

\subsection{Enumerating Triangulations}

\paragraph*{Plantri}
Having the graph generator plantri \cite{BrinkmannMcKay1999}\footnote{Plantri is available from \url{https://users.cecs.anu.edu.au/~bdm/plantri/}. 
For more information on plantri, 
we refer to \url{https://users.cecs.anu.edu.au/~bdm/plantri/plantri-guide.txt},
and for more information on the graph6 format, 
we refer to \url{https://users.cecs.anu.edu.au/~bdm/data/formats.html}.}  
installed,
it can be run with parameters ``[number of points] -g'', 
to enumerate all triangulations on the specified number of points in graph6 format.
With the additional parameter ``-c4'', only 4-connected triangulations are enumerated.
As an example, following command enumerates all 4-connected triangulations on 8 vertices:
\begin{verbatim}
$ ./plantri 8 -g -c4 
./plantri 8 -g -c4 
G|tJH{
G|thXs
2 triangulations written to stdout; cpu=0.00 sec
\end{verbatim}
To store graphs in files, one can simply pipe the standard output to the desired file:
\begin{verbatim}
$ ./plantri 8 -g -c4 > n8c4.g6
./plantri 8 -g -c4 
2 triangulations written to stdout; cpu=0.00 sec
\end{verbatim}

\paragraph*{Filter Triangulations}
Having the mathematics software system SageMath \cite{sagemath_website} installed,
we used the Sage-scripts \verb|filter_3tree.sage| and  \verb|filter_maxdeg.sage|
to filter stacked triangulations and triangulations with maximum degree $|V|-1$, respectively.

The script \verb|filter_3tree_relabel.sage| is a slight modification of \verb|filter_3tree.sage|,
which relabels the vertices of the given graph in a way, 
such that the vertices 0, 1, and 2 span the initial triangle,
and the $k$-th vertex is stacked into a triangular face of the subgraph 
induced by the vertices $0,1,\ldots,k-1$. 

Note that the triangulations enumerated by plantri do not necessarily fulfill this property. 
The triangulations shown in Listings~\ref{lst:11triangs_1} and~\ref{lst:11triangs_2}
were relabed using \verb|filter_3tree_relabel.sage|.
The respective files are available in the folder \verb|data/triangulations/|.

The scripts \verb|filter_property1.sage| and \verb|filter_property2.sage| 
can be used to filter triangulations 
which fulfill Properties~1 and~2 from Section~\ref{ssec:QuantitiveApproach}.

\paragraph*{Edge-List Encoding}
We have chosen a different plain text format, which is easier to load in C++:
We encode a graph by its edge list.
For each graph, we write the start and end vertices of the edges 
$\{u_1,v_1\},\ldots,\{u_m,v_m\}$
simply as \mbox{``$u_1$ $v_1$ $u_2$ $v_2$ \ldots $u_m$ $v_m$''} in a line, followed by a line-break.
The following example gives an illustration
(continues with the 4-connected 8-vertex triangulations from before):
\begin{verbatim}
$ sage encode.sage n8c4.g6 
0 1 0 2 0 3 0 4 1 2 1 4 1 5 1 6 2 3 2 6 2 7 3 4 3 7 4 5 4 7 5 6 5 7 6 7
0 1 0 2 0 3 0 4 1 2 1 4 1 5 2 3 2 5 2 6 2 7 3 4 3 7 4 5 4 6 4 7 5 6 6 7
\end{verbatim}
This encoding can be performed using the Sage-script \verb|encode.sage|.

\paragraph*{Drawing Graphs}
Last but not least, we provide the script \verb|draw.sage|
which we used to automatically generate drawings for stacked triangulations; 
see Listings~\ref{lst:11triangs_1} and~\ref{lst:11triangs_2}.
The idea is to start with a Tutte embedding 
and then use global optimization heuristics\footnote{Cf.\ \url{http://docs.scipy.org/doc/scipy/reference/generated/scipy.optimize.minimize.html}} 
to optimize a certain quality function, 
which simultaneously maximizes edge lengths and vertex-edge distances.

\subsection{Testing $n$-Universality}

\paragraph*{test\_universal\_sets}
We provide a C++ program \verb|test_universal_sets| to find $n$-universal point sets.
The program (see the folder \verb|cprogram/scripts/test_universal_sets/|)
can be built analogously to \verb|extend_order_type| (using qmake and make), except that
Minisat is required to be build as a library first.

\paragraph*{Building the Minisat Library}
As described in the README file delivered with Minisat \cite{EenSorensen2003}\footnote{Available from \url{http://minisat.se/MiniSat.html}}
(cf.~\verb|cprogram/minisat-2.2.0|),
one needs to set the MROOT variable. 
This can be done for example with the following command:
\begin{verbatim}
$ export MROOT=$PWD
\end{verbatim}
In the \verb|simp| folder from Minisat, one then can run 
\begin{verbatim}
$ make lib_release.a 
\end{verbatim}
to build the library. 
Having the \verb|lib_release.a| built, we are now ready to build our program \verb|test_universal_sets|.
Note that, if minisat-2.2.0 is not placed inside the basis directory (where the \verb|otlib.pro| is located), 
one might need to slightly adapt the project file \verb|test_universal_sets.pro| so 
that the minisat headers and library are found. 
In particular, the following two lines might need to be adapted:
\begin{verbatim}
INCLUDEPATH += $$OTLIBDIR/minisat-2.2.0
LIBS += $$OTLIBDIR/minisat-2.2.0/simp/lib_release.a
\end{verbatim}

\paragraph*{Usage of the Program}
\begin{verbatim}
./test_universal_sets [n] [order types file] [graphs file] [phase]
                      [parts] [from part] [to part]
\end{verbatim}
The parameters can be described as follows:
\begin{itemize}
 \item ``n'' is the number of points in the abstract order types from the input file,
 \item ``order types file'' is the path to the input file for abstract order types,
 \item ``graphs file'' is the path to the input file for graphs,
 \item ``phase'' specifies the actions which should be performed (see below),
 \item ``parts'' is the number of threads in total,
 \item ``from part'' is the id of the first thread to be run, and
 \item ``to part'' is the id of the first thread not to be run.
\end{itemize}
Abstract order types are again encoded by their small lambda matrix as described above.
For the graph file we have chosen the plain-text edge-list format described above.
The phase parameter describes, what the program should test:
\begin{itemize}
\item 
If the program is ran with parameter ``phase=1'', 
then point sets are filtered out which do not fulfill 
the necessary conditions described in Phase~1 of Section~\ref{sec:proof_conflict_11}.

\item 
If the program is ran with parameter ``phase=2'', 
then for each point set all graphs from the list are tested for simultaneous embeddability.
As described in Phase~2 of Section~\ref{sec:proof_conflict_11},
we stop as soon as one graph is not embeddable, 
and we use a priority queue to speedup the computations.

The parameter ``phase=2'' is also used to test Phase~4 of Section~\ref{sec:proof_conflict_11}.

\item 
If the program is ran with parameter ``phase=3'',  
then all pairs of order types and graphs are tested for embeddability.
Unlike for the computations of Phase~2, 
we do not change the order of the list of graphs.
For each given order type from the input file, 
a line of zeros and ones is written to a plain-text output file,
where the $j$-th zero/one in the $i$-th line encodes 
whether the $j$-th graph can be embedded on the $i$-th point set.
In the following we refer to this file as ``stat''-file.

The parameter ``phase=3''  is also used to test Phase~5 of Section~\ref{sec:proof_conflict_11}.
\end{itemize}

Note that for Phase~6 of Section~\ref{sec:proof_conflict_11} 
one can simply concatenate the stat-files obtained in Phases~3 and~5, and run CPLEX/Gurobi 
-- no additional computations with our C++ program are necessary.

\paragraph*{Loading Realizations from the Order Type Database}

It is also possible to load files from the order type database \cite{AichholzerOTDB}, 
which provide point set realizations of all order types.
To do so, one simply needs to change the line
\begin{verbatim}
// #define REALIZATIONS
\end{verbatim}
to
\begin{verbatim}
#define REALIZATIONS
\end{verbatim}
in the source file \verb|test_universal_sets.cpp|.
Note that 
in the binary files \verb|otypes04.b08|, \ldots, \verb|otypes08.b08| (available at \cite{AichholzerOTDB})
each order type of $n = 4,\ldots, 8$ points is encoded by one of its realizing point sets: 
the points $(x_1,y_1),\ldots,(x_n,y_n)$
are encoded as ``$x_1 y_1 \ldots x_n y_n$'' using 1 byte per coordinate (values inbetween 0 and 255).
For $n=9$, $n=10$, and $n=11$, 
each coordinate is encoded using 2 bytes  (values inbetween 0 and 65536).

\subsection{Integer Programming}

\paragraph*{Gurobi}
Having Gurobi/Gurobipy  installed 
\cite{gurobi}\footnote{Cf.\ Section 12 ``Python Interface'' 
from \url{http://www.gurobi.com/documentation/8.1/quickstart_linux.pdf}}\footnote{For information on free academic versions, 
checkout \url{http://www.gurobi.com/academia/academia-center}},
we used the script \verb|test_min_cover.py| to create 
a (Mixed) Integer Linear Programming instance
from a stat-file (created from our C++ program).
The script parses the input file, 
writes the instance to an ``lp'' file,
and then starts the Gurobi solver to find an (optimal) solution.
An instance, which is stored in an lp-file, 
can also be read and solved on a different machine for example via the following command:
\begin{verbatim}
$ gurobi
...
gurobi> m=read("n11_phase5_statistic_stacked.txt.instance.lp")
gurobi> m.optimize()
Optimize a model with 17533 rows, 423 columns and 1031205 nonzeros
...
Explored 295 nodes (25169 simplex iterations) in 69.54 seconds
Thread count was 6 (of 6 available processors)

Solution count 5: 12 13 14 ... 50

Optimal solution found (tolerance 1.00e-04)
Best objective 1.200000000000e+01, best bound 1.200000000000e+01, gap 0.0000%
\end{verbatim}

It is worth to note that, when solving an instance using Gurobi (also with CPLEX), 
the current upper and lower bound on the optimal value is printed to the console every few seconds.
Moreover, when aborting the solving process (CTRL+C), the currently best solution is printed.

\paragraph*{CPLEX}

An instance, which is stored in an lp-file, can be read and solved
via the CPLEX Interactive Optimizer \cite{IBMstudio}\footnote{
For information on free academic versions, 
checkout \url{http://www.ibm.com/developerworks/community/blogs/jfp/entry/CPLEX_Is_Free_For_Students}
} as exemplified in the following:

\begin{verbatim}
CPLEX> read instance.lp
...
CPLEX> optimize
...
\end{verbatim}

\section{The Conflict Collection}
\label{sec:conflict_collection}
\begin{lstlisting}[
caption={Edge-lists of the 11 stacked triangulations from collection~$\GG$ obtained in Phase~3.},
captionpos=b,label=lst:11triangs_1,basicstyle=\normalsize,breaklines=true]
G_1 = [(0,1),(0,2),(0,3),(0,4),(0,5),(0,6),(0,7),(0,8),(1,2),(1,7),(1,8),(1,9),(2,3),(2,5),(2,6),(2,7),(2,9),(2,10),(3,4),(3,5),(3,10),(4,5),(5,6),(5,10),(6,7),(7,8),(7,9)]

G_2 = [(0,1),(0,2),(0,3),(0,4),(0,5),(0,6),(0,7),(0,8),(1,2),(1,7),(1,8),(1,9),(2,3),(2,6),(2,7),(2,9),(3,4),(3,5),(3,6),(3,10),(4,5),(5,6),(5,10),(6,7),(6,10),(7,8),(7,9)]

G_3 = [(0,1),(0,2),(0,3),(0,4),(0,5),(0,6),(0,7),(0,8),(1,2),(1,7),(1,8),(2,3),(2,6),(2,7),(2,9),(2,10),(3,4),(3,5),(3,6),(3,10),(4,5),(5,6),(6,7),(6,9),(6,10),(7,8),(7,9)]

G_4 = [(0,1),(0,2),(0,3),(0,4),(0,5),(0,6),(0,7),(0,8),(0,9),(1,2),(1,7),(1,9),(2,3),(2,6),(2,7),(2,10),(3,4),(3,6),(4,5),(4,6),(5,6),(6,7),(6,10),(7,8),(7,9),(7,10),(8,9)]

G_5 = [(0,1),(0,2),(0,3),(0,4),(0,5),(0,6),(0,7),(0,8),(1,2),(1,3),(1,4),(1,8),(1,9),(1,10),(2,3),(3,4),(4,5),(4,8),(4,10),(5,6),(5,8),(6,7),(6,8),(7,8),(8,9),(8,10),(9,10)]

G_6 = [(0,1),(0,2),(0,3),(0,4),(0,5),(0,6),(0,7),(0,8),(1,2),(1,8),(2,3),(2,4),(2,5),(2,8),(3,4),(4,5),(5,6),(5,8),(5,9),(6,7),(6,8),(6,9),(6,10),(7,8),(8,9),(8,10),(9,10)]

G_7 = [(0,1),(0,2),(0,3),(0,4),(0,5),(0,6),(0,7),(0,8),(1,2),(1,4),(1,5),(1,7),(1,8),(1,9),(1,10),(2,3),(2,4),(3,4),(4,5),(4,10),(5,6),(5,7),(5,10),(6,7),(7,8),(7,9),(8,9)]

G_8 = [(0,1),(0,2),(0,3),(0,4),(0,5),(0,6),(0,7),(0,8),(1,2),(1,4),(1,5),(1,7),(1,8),(1,9),(2,3),(2,4),(3,4),(4,5),(4,9),(5,6),(5,7),(5,9),(5,10),(6,7),(6,10),(7,8),(7,10)]

G_9 = [(0,1),(0,2),(0,3),(0,4),(0,5),(0,6),(0,7),(0,8),(1,2),(1,3),(1,5),(1,7),(1,8),(1,9),(2,3),(2,9),(3,4),(3,5),(3,9),(3,10),(4,5),(4,10),(5,6),(5,7),(5,10),(6,7),(7,8)]

G_10 = [(0,1),(0,2),(0,3),(0,4),(0,5),(0,6),(0,7),(0,8),(1,2),(1,3),(1,5),(1,7),(1,8),(1,9),(1,10),(2,3),(3,4),(3,5),(3,10),(4,5),(5,6),(5,7),(5,9),(5,10),(6,7),(7,8),(7,9)]

G_11 = [(0,1),(0,2),(0,3),(0,4),(0,5),(0,6),(0,7),(0,8),(1,2),(1,3),(1,4),(1,7),(1,8),(1,9),(2,3),(3,4),(4,5),(4,6),(4,7),(5,6),(6,7),(7,8),(7,9),(7,10),(8,9),(8,10),(9,10)]

G_12 = [(0,1),(0,2),(0,3),(0,4),(0,5),(0,6),(0,7),(0,8),(1,2),(1,4),(1,5),(1,6),(1,8),(2,3),(2,4),(2,9),(3,4),(3,9),(4,5),(4,9),(5,6),(6,7),(6,8),(6,10),(7,8),(7,10),(8,10)]

G_13 = [(0,1),(0,2),(0,3),(0,4),(0,5),(0,6),(0,7),(0,8),(1,2),(1,3),(1,5),(1,8),(1,9),(2,3),(3,4),(3,5),(3,9),(4,5),(5,6),(5,8),(5,9),(5,10),(6,7),(6,8),(6,10),(7,8),(8,10)]

G_14 = [(0,1),(0,2),(0,3),(0,4),(0,5),(0,6),(0,7),(0,8),(1,2),(1,4),(1,5),(1,7),(1,8),(1,9),(2,3),(2,4),(2,10),(3,4),(3,10),(4,5),(4,10),(5,6),(5,7),(5,9),(6,7),(7,8),(7,9)]

G_15 = [(0,1),(0,2),(0,3),(0,4),(0,5),(0,6),(0,7),(1,2),(1,5),(1,7),(1,8),(2,3),(2,4),(2,5),(3,4),(4,5),(5,6),(5,7),(5,8),(5,9),(6,7),(7,8),(7,9),(7,10),(8,9),(8,10),(9,10)]

G_16 = [(0,1),(0,2),(0,3),(0,4),(0,5),(0,6),(1,2),(1,6),(2,3),(2,6),(3,4),(3,5),(3,6),(3,7),(4,5),(4,7),(4,8),(4,9),(4,10),(5,6),(5,7),(5,8),(5,9),(5,10),(7,8),(8,9),(9,10)]

G_17 = [(0,1),(0,2),(0,3),(0,4),(0,5),(0,6),(0,7),(0,8),(1,2),(1,3),(1,5),(1,7),(1,8),(1,9),(2,3),(3,4),(3,5),(3,10),(4,5),(4,10),(5,6),(5,7),(5,10),(6,7),(7,8),(7,9),(8,9)]

G_18 = [(0,1),(0,2),(0,3),(0,4),(0,5),(0,6),(0,7),(0,8),(0,9),(1,2),(1,8),(1,9),(2,3),(2,8),(3,4),(3,8),(4,5),(4,8),(5,6),(5,7),(5,8),(5,10),(6,7),(7,8),(7,10),(8,9),(8,10)]

G_19 = [(0,1),(0,2),(0,3),(0,4),(0,5),(0,6),(0,7),(0,8),(0,9),(1,2),(1,8),(1,9),(2,3),(2,4),(2,5),(2,8),(2,10),(3,4),(4,5),(5,6),(5,7),(5,8),(5,10),(6,7),(7,8),(8,9),(8,10)]

G_20 = [(0,1),(0,2),(0,3),(0,4),(0,5),(0,6),(0,7),(1,2),(1,3),(1,4),(1,7),(1,8),(2,3),(2,8),(2,9),(3,4),(3,8),(3,9),(3,10),(4,5),(4,7),(5,6),(5,7),(6,7),(8,9),(8,10),(9,10)]

G_21 = [(0,1),(0,2),(0,3),(0,4),(0,5),(0,6),(1,2),(1,6),(1,7),(1,8),(1,9),(2,3),(2,6),(2,7),(2,9),(2,10),(3,4),(3,5),(3,6),(4,5),(5,6),(6,7),(7,8),(7,9),(7,10),(8,9),(9,10)]

G_22 = [(0,1),(0,2),(0,3),(0,4),(0,5),(0,6),(0,7),(0,8),(1,2),(1,6),(1,8),(2,3),(2,4),(2,5),(2,6),(2,9),(2,10),(3,4),(4,5),(5,6),(5,10),(6,7),(6,8),(6,9),(6,10),(7,8),(9,10)]

G_23 = [(0,1),(0,2),(0,3),(0,4),(0,5),(0,6),(0,7),(0,8),(1,2),(1,8),(2,3),(2,4),(2,5),(2,8),(2,9),(2,10),(3,4),(4,5),(5,6),(5,8),(5,10),(6,7),(6,8),(7,8),(8,9),(8,10),(9,10)]

G_24 = [(0,1),(0,2),(0,3),(0,4),(0,5),(0,6),(0,7),(0,8),(0,9),(1,2),(1,6),(1,8),(1,9),(1,10),(2,3),(2,4),(2,6),(2,10),(3,4),(4,5),(4,6),(5,6),(6,7),(6,8),(6,10),(7,8),(8,9)]

G_25 = [(0,1),(0,2),(0,3),(0,4),(0,5),(0,6),(0,7),(0,8),(1,2),(1,8),(1,9),(2,3),(2,4),(2,5),(2,8),(2,9),(2,10),(3,4),(4,5),(5,6),(5,8),(5,10),(6,7),(6,8),(7,8),(8,9),(8,10)]

G_26 = [(0,1),(0,2),(0,3),(0,4),(0,5),(0,6),(0,7),(0,8),(0,9),(1,2),(1,4),(1,9),(2,3),(2,4),(3,4),(4,5),(4,6),(4,7),(4,9),(4,10),(5,6),(6,7),(7,8),(7,9),(7,10),(8,9),(9,10)]

G_27 = [(0,1),(0,2),(0,3),(0,4),(0,5),(0,6),(0,7),(0,8),(1,2),(1,3),(1,4),(1,8),(1,9),(2,3),(3,4),(4,5),(4,6),(4,7),(4,8),(4,9),(4,10),(5,6),(6,7),(7,8),(8,9),(8,10),(9,10)]
\end{lstlisting}


\begin{lstlisting}[
caption={Edge-lists of the 12 stacked triangulations from collection~$\HH$ obtained in Phase~5.},
captionpos=b,label=lst:11triangs_2,basicstyle=\normalsize,breaklines=true]
H_1 = [(0,1),(0,2),(0,3),(0,4),(0,5),(0,6),(0,7),(0,8),(0,9),(1,2),(1,6),(1,9),(1,10),(2,3),(2,6),(2,10),(3,4),(3,5),(3,6),(4,5),(5,6),(6,7),(6,9),(6,10),(7,8),(7,9),(8,9)]

H_2 = [(0,1),(0,2),(0,3),(0,4),(0,5),(0,6),(0,7),(1,2),(1,3),(1,4),(1,5),(1,7),(1,8),(2,3),(3,4),(4,5),(5,6),(5,7),(5,8),(5,9),(5,10),(6,7),(6,10),(7,8),(7,9),(7,10),(8,9)]

H_3 = [(0,1),(0,2),(0,3),(0,4),(0,5),(0,6),(0,7),(0,8),(0,9),(0,10),(1,2),(1,3),(1,8),(1,10),(2,3),(3,4),(3,5),(3,8),(4,5),(5,6),(5,8),(6,7),(6,8),(7,8),(8,9),(8,10),(9,10)]

H_4 = [(0,1),(0,2),(0,3),(0,4),(0,5),(0,6),(0,7),(0,8),(0,9),(0,10),(1,2),(1,3),(1,4),(1,10),(2,3),(3,4),(4,5),(4,7),(4,10),(5,6),(5,7),(6,7),(7,8),(7,9),(7,10),(8,9),(9,10)]

H_5 = [(0,1),(0,2),(0,3),(0,4),(0,5),(0,6),(0,7),(0,8),(0,9),(0,10),(1,2),(1,6),(1,7),(1,9),(1,10),(2,3),(2,4),(2,6),(3,4),(4,5),(4,6),(5,6),(6,7),(7,8),(7,9),(8,9),(9,10)]

H_6 = [(0,1),(0,2),(0,3),(0,4),(0,5),(0,6),(0,7),(0,8),(0,9),(1,2),(1,4),(1,9),(1,10),(2,3),(2,4),(3,4),(4,5),(4,6),(4,7),(4,9),(4,10),(5,6),(6,7),(7,8),(7,9),(8,9),(9,10)]

H_7 = [(0,1),(0,2),(0,3),(0,4),(0,5),(0,6),(0,7),(0,8),(0,9),(0,10),(1,2),(1,6),(1,7),(1,8),(1,10),(2,3),(2,6),(3,4),(3,5),(3,6),(4,5),(5,6),(6,7),(7,8),(8,9),(8,10),(9,10)]

H_8 = [(0,1),(0,2),(0,3),(0,4),(0,5),(0,6),(0,7),(0,8),(0,9),(0,10),(1,2),(1,3),(1,10),(2,3),(3,4),(3,5),(3,6),(3,7),(3,8),(3,9),(3,10),(4,5),(5,6),(6,7),(7,8),(8,9),(9,10)]

H_9 = [(0,1),(0,2),(0,3),(0,4),(0,5),(0,6),(0,7),(0,8),(0,9),(0,10),(1,2),(1,5),(1,6),(1,10),(2,3),(2,4),(2,5),(3,4),(4,5),(5,6),(6,7),(6,9),(6,10),(7,8),(7,9),(8,9),(9,10)]

H_10 = [(0,1),(0,2),(0,3),(0,4),(0,5),(0,6),(0,7),(0,8),(0,9),(0,10),(1,2),(1,9),(1,10),(2,3),(2,4),(2,5),(2,8),(2,9),(3,4),(4,5),(5,6),(5,8),(6,7),(6,8),(7,8),(8,9),(9,10)]

H_11 = [(0,1),(0,2),(0,3),(0,4),(0,5),(0,6),(0,7),(0,8),(0,9),(0,10),(1,2),(1,5),(1,8),(1,9),(1,10),(2,3),(2,4),(2,5),(3,4),(4,5),(5,6),(5,8),(6,7),(6,8),(7,8),(8,9),(9,10)]

H_12 = [(0,1),(0,2),(0,3),(0,4),(0,5),(0,6),(0,7),(0,8),(0,9),(1,2),(1,4),(1,5),(1,7),(1,8),(1,9),(2,3),(2,4),(3,4),(4,5),(5,6),(5,7),(5,10),(6,7),(6,10),(7,8),(7,10),(8,9)]

H_13 = [(0,1),(0,2),(0,3),(0,4),(0,5),(0,6),(0,7),(0,8),(1,2),(1,3),(1,4),(1,7),(1,8),(2,3),(3,4),(4,5),(4,6),(4,7),(4,9),(4,10),(5,6),(5,10),(6,7),(6,9),(6,10),(7,8),(7,9)]

H_14 = [(0,1),(0,2),(0,3),(0,4),(0,5),(0,6),(0,7),(0,8),(0,9),(0,10),(1,2),(1,3),(1,5),(1,7),(1,9),(1,10),(2,3),(3,4),(3,5),(4,5),(5,6),(5,7),(6,7),(7,8),(7,9),(8,9),(9,10)]

H_15 = [(0,1),(0,2),(0,3),(0,4),(0,5),(0,6),(0,7),(0,8),(0,9),(0,10),(1,2),(1,10),(2,3),(2,10),(3,4),(3,5),(3,10),(4,5),(5,6),(5,7),(5,8),(5,10),(6,7),(7,8),(8,9),(8,10),(9,10)]

H_16 = [(0,1),(0,2),(0,3),(0,4),(0,5),(0,6),(0,7),(0,8),(0,9),(0,10),(1,2),(1,10),(2,3),(2,10),(3,4),(3,6),(3,10),(4,5),(4,6),(5,6),(6,7),(6,8),(6,10),(7,8),(8,9),(8,10),(9,10)]

H_17 = [(0,1),(0,2),(0,3),(0,4),(0,5),(0,6),(0,7),(0,8),(0,9),(0,10),(1,2),(1,3),(1,4),(1,5),(1,6),(1,10),(2,3),(3,4),(4,5),(5,6),(6,7),(6,8),(6,9),(6,10),(7,8),(8,9),(9,10)]

H_18 = [(0,1),(0,2),(0,3),(0,4),(0,5),(0,6),(0,7),(0,8),(1,2),(1,3),(1,6),(1,8),(1,9),(2,3),(3,4),(3,5),(3,6),(3,9),(3,10),(4,5),(5,6),(5,10),(6,7),(6,8),(6,9),(6,10),(7,8)]

H_19 = [(0,1),(0,2),(0,3),(0,4),(0,5),(0,6),(0,7),(0,8),(1,2),(1,5),(1,6),(1,7),(1,8),(1,9),(2,3),(2,4),(2,5),(3,4),(4,5),(5,6),(6,7),(7,8),(7,9),(7,10),(8,9),(8,10),(9,10)]

H_20 = [(0,1),(0,2),(0,3),(0,4),(0,5),(0,6),(0,7),(0,8),(0,9),(0,10),(1,2),(1,4),(1,6),(1,8),(1,10),(2,3),(2,4),(3,4),(4,5),(4,6),(5,6),(6,7),(6,8),(7,8),(8,9),(8,10),(9,10)]

H_21 = [(0,1),(0,2),(0,3),(0,4),(0,5),(0,6),(0,7),(0,8),(1,2),(1,5),(1,6),(1,8),(1,9),(2,3),(2,4),(2,5),(2,10),(3,4),(3,10),(4,5),(4,10),(5,6),(6,7),(6,8),(6,9),(7,8),(8,9)]

H_22 = [(0,1),(0,2),(0,3),(0,4),(0,5),(0,6),(0,7),(0,8),(1,2),(1,4),(1,5),(1,7),(1,8),(2,3),(2,4),(2,9),(2,10),(3,4),(3,10),(4,5),(4,9),(4,10),(5,6),(5,7),(6,7),(7,8),(9,10)]
\end{lstlisting}

\section{Discussion}
\label{sec:discussion}

In Section~\ref{sec:proof_theorem_lower129}, we provided an improved lower bound for $f_p(n)$ and $f_s(n)$.
However, the best known general upper bounds remain far from linear.

\medskip

In Section~\ref{sec:proof_conflict_11}, 
we have applied the ideas from Phases~2 and~3 twice (cf. Phases~4 and~5)
to reduce the size of a conflict collection.
One could further proceed with this strategy to find even smaller conflict collections (if such exist).
Also one could simply test 
whether all elements from the conflict collection are indeed necessary,
or whether certain elements can be removed.
Note that, to compute a minimal conflict collection for $n=11$, one could theoretically
check which graphs admit an embedding on which point set and then find a minimal set cover 
as described in Phase~3 (Section~\ref{sec:proof_conflict_11}).
In practice, however, formulating such a minimal set cover instance (as integer program) is not reasonable
because testing the embeddability of every graph in every point set would be an extremely time consuming task.
(Recall that we used a priority queue to speed up our computation, so only a few pairs were actually tested.
Also recall that, to generate the set cover instances, 
we only looked at a comparably small number of order types.)
And even if such an instance was formulated, 
due to its size, the IP/set cover might not be solvable optimally in reasonable time.

\medskip

Besides the computations for $n=11$ points,
we also adapted our program to
find all $n$-universal order types on $n$ points for every $n \le 10$, and hence could
verify the results from \cite[Table~1]{CardinalHoffmannKusters2015}.
To be precise, we found 5,956 9-universal abstract order types on $n=9$ points,
whereas only 5,955 are realizable as point sets.
It is worth to note that 
there is exactly one non-realizable abstract order type on 9 points in the projective plane, 
which is dual to the simple non-Pappus arrangement,
and that all abstract order types on $n \le 8$ points are realizable. 
Besides the already known 2,072 realizable order types on 10 points,
no further non-realizable 10-universal abstract order types were found.
For more details on realizability see for example~\cite{Krasser2003} or~\cite{FelsnerGoodman2016}.

Unfortunately, we do not have an argument for subsets/supersets of $n$-universal point sets,
and thus the question for $n=12,13,14$ remains open.
However, based on computational evidence (see also \cite[Table~1]{CardinalHoffmannKusters2015}), 
we strongly conjecture that no $n$-universal set of $n$ points exists for $n \ge 11$.

As mentioned in the introduction of this paper,
various graph classes have been studied for this problem.
Even though our contribution on 4-connected planar graphs in Section~\ref{sec:4connected} is rather small,
it gives some evidence that comparably less points are needed to embed 4-connected planar graphs.
In fact, we would not be surprised if $n$-universal sets of $n$ points exist for 4-connected planar graphs.

Last but not least, 
we want to mention that 
universal sets indeed must have a very special structure:
It is not hard to see that, for embedding nested triangulations on a grid,
at least $\Omega(n) \times \Omega(n)$ points are required.
Quite recently, this bound was generalized to ``random'' point sets 
(points chosen uniformly and independently at random from the unit-square) \cite{ChoiChrobakCostello2019}. Therefore the probabilistic method in its basic form will not succeed in proving subquadratic upper bounds on $f_p(n)$.

\subsection{The Certifying SAT Model and the Bug}
\label{sec:thebug}

When we started investigating universal point sets,
we first formulated a SAT instance
to find an abstract order type on $n$ points  
which is $n$-universal (cf.\ \verb|sat_test.sage|).
For $ n\le 10$, the solver almost instantly found an $n$-universal set of $n$ points,
however, for $n = 11$ the program did not terminate.
Therefore, we had to come up with a slightly more complicated procedure involving some C++ code 
(cf.\ Section~\ref{ssec:QuantitiveApproach}).

When preparing this full version,
we modified the original SAT instance 
to only test the ``conflict set'' of 23 stacked triangulations 
from earlier versions of this paper \cite{3S_eurocg_version,3S_gd_version},
so that we have an independent computer-proof verifying the correctness.
However, the SAT solver almost instantly found an abstract order type 
which was universal for that ``conflict set''.

We located and fixed a small bug in the C++ source,
and after re-running all computations, 
we ended up with the conflict set of 49 stacked triangulations.
An independent SAT model to verify this result provides good evidence, 
as it did not manage to produce a solution to the SAT instance within a reasonable amount of time (we had the program running for several weeks).

\section*{Acknowledgements}

Manfred Scheucher was supported by DFG Grant FE~340/12-1.
Hendrik Schrezenmaier was supported by DFG Grant FE-340/11-1.
Raphael Steiner was supported by DFG-GRK 2434.

\bibliographystyle{alphaabbrv-url}
\bibliography{refs}

\newcommand{\etalchar}[1]{$^{#1}$}
\begin{thebibliography}{ABDB{\etalchar{+}}18}

\bibitem[AAK02]{AichholzerAurenhammerKrasser2001}
O.~Aichholzer, F.~Aurenhammer, and H.~Krasser.
\newblock \href{http://doi.org/10.1023/A:1021231927255}{{Enumerating Order
  Types for Small Point Sets with Applications}}.
\newblock {\em Order}, 19(3):265--281, 2002.

\bibitem[ABDB{\etalchar{+}}18]{Angelini2018}
P.~Angelini, T.~Bruckdorfer, G.~Di~Battista, M.~Kaufmann, T.~Mchedlidze,
  V.~Roselli, and C.~Squarcella.
\newblock \href{http://doi.org/10.1007/s00454-018-0009-x}{Small universal point
  sets for k-outerplanar graphs}.
\newblock {\em Discrete \& Computational Geometry}, pages 1--41, 2018.

\bibitem[Aic]{AichholzerOTDB}
O.~Aichholzer.
\newblock {Enumerating Order Types for Small Point Sets with Applications}.
\newblock \\
  \url{http://www.ist.tugraz.at/aichholzer/research/rp/triangulations/ordertypes/}.

\bibitem[AK06]{AichholzerKrasser2006}
O.~Aichholzer and H.~Krasser.
\newblock \href{http://doi.org/10.1016/j.comgeo.2005.07.005}{{Abstract Order
  Type Extension and New Results on the Rectilinear Crossing Number}}.
\newblock {\em Computational Geometry: Theory and Applications}, 36(1):2--15,
  2006.

\bibitem[BCD{\etalchar{+}}07]{BrassCDEEIKLM2007}
P.~Brass, E.~Cenek, C.~A. Duncan, A.~Efrat, C.~Erten, D.~P. Ismailescu, S.~G.
  Kobourov, A.~Lubiw, and J.~S. Mitchell.
\newblock \href{http://doi.org/10.1016/j.comgeo.2006.05.006}{On simultaneous
  planar graph embeddings}.
\newblock {\em Computational Geometry}, 36(2):117--130, 2007.

\bibitem[BCDE14]{BannisterCDE2014}
M.~J. Bannister, Z.~Cheng, W.~E. Devanny, and D.~Eppstein.
\newblock \href{http://doi.org/10.7155/jgaa.00318}{{Superpatterns and Universal
  Point Sets}}.
\newblock {\em Journal of Graph Algorithms and Applications}, 18(2):177--209,
  2014.

\bibitem[BFK15]{BalkoFulekKyncl2015}
M.~Balko, R.~Fulek, and J.~Kyn{\v{c}}l.
\newblock \href{http://doi.org/10.1007/s00454-014-9644-z}{{Crossing Numbers and
  Combinatorial Characterization of Monotone Drawings of $K_n$}}.
\newblock {\em Discrete {\&} Computational Geometry}, 53(1):107--143, 2015.

\bibitem[BM99]{BrinkmannMcKay1999}
G.~Brinkmann and B.~D. McKay.
\newblock \href{http://doi.org/10.1016/S1571-0653(05)80016-2}{{Fast generation
  of some classes of planar graphs}}.
\newblock {\em Electronic Notes in Discrete Mathematics}, 3:28--31, 1999.

\bibitem[Bra08]{Brandenburg2008}
F.~J. Brandenburg.
\newblock \href{http://doi.org/10.1016/j.endm.2008.06.005}{Drawing planar
  graphs on $\frac{8}{9}n^2$ area}.
\newblock {\em Electronic Notes in Discrete Mathematics}, 31:37--40, 2008.

\bibitem[{Cab}06]{Cabello2006}
S.~{Cabello}.
\newblock \href{http://doi.org/10.7155/jgaa.00132}{Planar embeddability of the
  vertices of a graph using a fixed point set is {NP}-hard}.
\newblock {\em Journal of Graph Algorithms and Applications}, 10(2):353--363,
  2006.

\bibitem[CCC19]{ChoiChrobakCostello2019}
A.~Choi, M.~Chrobak, and K.~Costello.
\newblock {An $\Omega(n^2)$ Lower Bound for Random Universal Sets for Planar
  Graphs}.
\newblock \href{http://arXiv.org/abs/1908.07097} {arXiv:1908.07097}, 2019.

\bibitem[CHK15]{CardinalHoffmannKusters2015}
J.~Cardinal, M.~Hoffmann, and V.~Kusters.
\newblock \href{http://doi.org/10.7155/jgaa.00374}{{On Universal Point Sets for
  Planar Graphs}}.
\newblock {\em Journal of Graph Algorithms and Applications}, 19(1):529--547,
  2015.

\bibitem[CK89]{Chrobak1989}
M.~Chrobak and H.~J. Karloff.
\newblock \href{http://doi.org/10.1145/74074.74088}{{A Lower Bound on the Size
  of Universal Sets for Planar Graphs}}.
\newblock {\em ACM SIGACT News}, 20(4):83--86, 1989.

\bibitem[CU96]{CastanedaUrrutia1996}
N.~Casta{\~n}eda and J.~Urrutia.
\newblock {Straight Line Embeddings of Planar Graphs on Point Sets}.
\newblock In {\em Proceedings of the 8th Canadian Conference on Computational
  Geometry (CCCG'96)}, pages 312--318, 1996.
\newblock \\ \url{http://www.cccg.ca/proceedings/1996/cccg1996_0052.pdf}.

\bibitem[DFPP90]{DeFraysseixPachPollack1990}
H.~De~Fraysseix, J.~Pach, and R.~Pollack.
\newblock \href{http://doi.org/10.1007/BF02122694}{How to draw a planar graph
  on a grid}.
\newblock {\em Combinatorica}, 10(1):41--51, 1990.

\bibitem[ES03]{EenSorensen2003}
N.~E{\'{e}}n and N.~S{\"{o}}rensson.
\newblock \href{http://doi.org/10.1007/978-3-540-24605-3\_37}{An extensible
  {SAT}-solver}.
\newblock In {\em Proceedings of Theory and Applications of Satisfiability
  Testing - {SAT} 2003}, pages 502--518, 2003.

\bibitem[FG18]{FelsnerGoodman2016}
S.~Felsner and J.~E. Goodman.
\newblock \href{http://doi.org/10.1201/9781315119601}{{Pseudoline
  Arrangements}}.
\newblock In Toth, O'Rourke, and Goodman, editors, {\em Handbook of Discrete
  and Computational Geometry}. CRC Press, 3 edition, 2018.

\bibitem[FT15]{FulekToth2013}
R.~Fulek and C.~D. T{\'o}th.
\newblock \href{http://doi.org/10.1016/j.jda.2014.12.005}{Universal point sets
  for planar three-trees}.
\newblock {\em Journal of Discrete Algorithms}, 30:101--112, 2015.

\bibitem[FW01]{FelsnerWeil2001}
S.~Felsner and H.~Weil.
\newblock \href{http://doi.org/10.1016/S0166-218X(00)00232-8}{{Sweeps,
  Arrangements and Signotopes}}.
\newblock {\em Discrete Applied Mathematics}, 109(1):67--94, 2001.

\bibitem[GP83]{GoodmanPollack1983}
J.~E. Goodman and R.~Pollack.
\newblock \href{http://doi.org/doi:10.1137/0212032}{{Multidimensional
  Sorting}}.
\newblock {\em SIAM Journal on Computing}, 12(3):484--507, 1983.

\bibitem[{Gur}18]{gurobi}
{Gurobi Optimization, LLC}.
\newblock {Gurobi Optimizer}, 2018.
\newblock \\ \url{http://www.gurobi.com}.

\bibitem[IBM18]{IBMstudio}
{IBM ILOG CPLEX Optimization Studio}, 2018.
\newblock \\ \url{http://www.ibm.com/products/ilog-cplex-optimization-studio/}.

\bibitem[Kra03]{Krasser2003}
H.~Krasser.
\newblock {\em {Order Types of Point Sets in the Plane}}.
\newblock PhD thesis, Institute for Theoretical Computer Science, Graz
  University of Technology, Austria, 2003.

\bibitem[Kur04]{Kurowski2004}
M.~Kurowski.
\newblock \href{http://doi.org/10.1016/j.ipl.2004.06.009}{{A $1.235n$ lower
  bound on the number of points needed to draw all $n$-vertex planar graphs}}.
\newblock {\em Information Processing Letters}, 92(2):95--98, 2004.

\bibitem[MP14]{McKayPiperno2014}
B.~D. McKay and A.~Piperno.
\newblock \href{http://doi.org/10.1016/j.jsc.2013.09.003}{{Practical graph
  isomorphism, II}}.
\newblock {\em Journal of Symbolic Computation}, 60:94--112, 2014.

\bibitem[PGMP91]{PachGritzmannMoharPollack1991}
J.~Pach, P.~Gritzmann, B.~Mohar, and R.~Pollack.
\newblock \href{http://doi.org/10.2307/2323956}{{Embedding a planar
  triangulation with vertices at specified points}}.
\newblock {\em American Mathematical Monthly}, 98:165--166, 1991.

\bibitem[S{\etalchar{+}}18a]{sagemath_website}
W.~A. Stein et~al.
\newblock {\em {S}age {M}athematics {S}oftware ({V}ersion 8.1)}.
\newblock The Sage Development Team, 2018.
\newblock \url{http://www.sagemath.org}.

\bibitem[S{\etalchar{+}}18b]{sagemath_graphtheory_manual}
W.~A. Stein et~al.
\newblock {\em {Sage Reference Manual: Graph Theory (Release 8.1)}}, 2018.
\newblock \\
  \url{http://doc.sagemath.org/pdf/en/reference/number_fields/number_fields.pdf}.

\bibitem[Sch]{scheucher_website}
M.~Scheucher.
\newblock {Webpage: Source Codes and Data for Universal Point Sets.}
\newblock \\
  \url{http://page.math.tu-berlin.de/~scheuch/supplemental/universal_sets}.

\bibitem[Sch90]{Schnyder1990}
W.~Schnyder.
\newblock {Embedding Planar Graphs on the Grid}.
\newblock In {\em Proceedings of the First Annual ACM-SIAM Symposium on
  Discrete Algorithms}, pages 138--148. Society for Industrial and Applied
  Mathematics, 1990.

\bibitem[Sch14]{scheucher2014}
M.~Scheucher.
\newblock {\em {On Order Types, Projective Classes, and Realizations}}.
\newblock Bachelor's thesis, Graz University of Technology, Austria, 2014.
\newblock \\
  \url{http://www.math.tu-berlin.de/~scheuch/publ/bachelors_thesis_tm_2014.pdf}.

\bibitem[SSS]{3S_gd_version}
M.~Scheucher, H.~Schrezenmaier, and R.~Steiner.
\newblock {A Note On Universal Point Sets for Planar Graphs}.
\newblock To appear in the Proc. of the 27th International Symposium on Graph
  Drawing and Network Visualization (GD'19).

\bibitem[SSS19]{3S_eurocg_version}
M.~Scheucher, H.~Schrezenmaier, and R.~Steiner.
\newblock \href{http://www.eurocg2019.uu.nl/papers/21.pdf}{{A Note On Universal
  Point Sets for Planar Graphs}}.
\newblock In {\em Proc. 35th European Workshop on Computational Geometry
  (EuroCG'19)}, pages 21:1--21:9, 2019.

\end{thebibliography}

\end{document}